\documentclass{IEEEtran}
\IEEEoverridecommandlockouts

\usepackage{cite}
\usepackage{amsmath,amssymb,amsfonts}
\usepackage{algorithmic}
\usepackage{graphicx}
\usepackage{textcomp}
\usepackage{xcolor}

\usepackage{amsthm}
\newtheoremstyle{theoremstyle}
  {1.5ex plus .2ex minus .2ex}
  {1.5ex plus .2ex minus .2ex}
  {\slshape}
  {}
  {\bfseries}
  {:}
  {.5em}
  {\thmname{#1}\thmnumber{ #2}\thmnote{ (#3)}}

\theoremstyle{theoremstyle}

\newtheorem{lemma}{Lemma}
\newtheorem*{proof*}{Proof}

\newtheorem{theorem}{Theorem}
\newtheorem{corollary}{Corollary}

\usepackage{float}
\usepackage{graphicx}
\usepackage[export]{adjustbox}  

\usepackage{subfigure}  

\usepackage{amsmath} 
\allowdisplaybreaks
\usepackage{enumitem}

\def\BibTeX{{\rm B\kern-.05em{\sc i\kern-.025em b}\kern-.08emT\kern-.1667em\lower.7ex\hbox{E}\kern-.125emX}}
\begin{document}

\title{Distributed Rotary Coverage Control of Multi-Agent Systems in Uncertain Environments \\}

\author{Chao Zhai, Yanlin Li \thanks{Chao Zhai and Yanlin Li are with School of Automation, China University of Geosciences, Wuhan 430074, China, and with Hubei Key Laboratory of Advanced Control and Intelligent Automation for Complex Systems and also with Engineering Research Center of Intelligent Technology for Geo-Exploration, Ministry of Education (email: zhaichao@amss.ac.cn).
The project was supported by the ``CUG Scholar" Scientific Research Funds 
at China University of Geosciences (Wuhan) (Project No.2020138).} 
}

\maketitle

\begin{abstract}
It is always a challenging task for multi-agent systems to achieve efficient and robust coverage in uncertain environments. The absence of global positioning information on the uncertain environment introduces significant complexity to the spatially distributed design of coverage control algorithms. 
To address this issue, this paper proposes a coverage control formulation 
based on beacon-free rotary pointer partition mechanism. A partition dynamics 
is designed to enable the asymptotical consensus of multi-agent reference points, as well as the workload-balanced subdivision of coverage region. On this basis, a distributed coverage control algorithm is developed to drive each agent toward the optimal deployment of their respective subregions, thereby minimizing the coverage cost. Simulation results demonstrate that the proposed coverage control method can significantly improve overall coverage efficiency with workload balance among agents, and exhibit strong adaptability and robustness in uncertain environments.
\end{abstract}


\section{Introduction}
In recent years, multi-agent coordination has seen remarkable progress and been widely applied in complex collective tasks such as search and rescue~\cite{b1}, autonomous exploration~\cite{b5}, area coverage~\cite{b3,b4}, border surveillance~\cite{b2} and environmental monitoring~\cite{b6}. In particular, coverage control is crucial for addressing many real-world missions.
The objective of multi-agent coverage control is to explore effective coordination and deployment strategies for multiple agents in complex environments, enabling them to access regions of interest, perform event monitoring, and optimize certain coverage performance indices.

At its core, the lack of global information and limitations in distributed coordination pose significant challenges to designing effective coverage control strategies in uncertain environments~\cite{b7}, where environment information is not available to each agent in advance. To overcome these challenges, space partitioning has emerged as a fundamental paradigm, enabling the decomposition of complex environments into manageable subregions and facilitating distributed implementation of coverage algorithms. 
Within this framework, facility location problems aim to optimally deploy mobile agents to handle random events and minimize service costs. Building upon this idea, Voronoi partition~\cite{b8} and equal workload partition~\cite{b9} provide effective solutions that guarantee the networked agents eventually converge to configurations maximizing overall coverage performance.
Voronoi partition and its extensions have long been fundamental tools in solving location optimization problems~\cite{b10,b11}.
Voronoi partition enables agents to divide the coverage area into distinct polygons by leveraging relative position information, while its centroidal Voronoi tessellation facilitates optimal resource placement~\cite{b12}. Geodesic Voronoi partition has been applied to deploy multi-agent systems in scenarios involving mixed-dimensional and heterogeneous environments~\cite{b10}.
Despite their widespread application in facility location planning, Voronoi-based approaches exhibit several limitations,
they suffer from several inherent limitations, such as high computational complexity, inefficient workload distribution, and even connectivity loss in nonconvex environments~\cite{b13}.
To overcome the aforementioned limitations, researchers have proposed the equal workload partition\cite{b3}, which divides the coverage region into subregions such that each agent is assigned an equal share of the sensing or service workload. 
Unlike traditional Voronoi-based methods that rely solely on geometric proximity, the equal workload partition explicitly incorporates environmental density or event frequency into the partitioning process, thereby achieving more balanced resource utilization among agents. In~\cite{b14}, a coverage formulation based on rotary pointer dynamics is developed to achieve workload balance through the continuous update of partition boundaries. 
However, this partition strategy depends on a pre-defined global reference point to generate the partition pointer, and thus its implementation is not fully distributed. Moreover, such reliance on global reference information limits its adaptability to scenarios where agents operate using only local sensing and relative position information, such as GPS-denied or communication-constrained environments.

To address these challenges, this paper presents a novel approach to multi-agent coverage control by leveraging rotary pointer partitions to enhance coverage efficiency in uncertain environments. The proposed rotary pointer partition enables multi-agent system to dynamically and consistently divide the coverage region into subregions without relying on a pre-defined global reference. This partition strategy ensures that each agent is responsible for its own subregion and can respond to random events while maintaining balanced workload across the entire region. 
As a result, the proposed coverage method is well-suited for scenarios where agents have access only to local information and relative position. In brief, key contributions of this study are summarized as follows: 1) A reference-point consensus strategy is developed to enable fully distributed partition without reliance on a global reference. 2) A distributed partition mechanism based on rotary pointer is designed to ensure workload balance among agents in uncertain environments. 3) A distributed control law is proposed to drive each agent toward the local optimum within its subregion, thereby improving overall coverage performance.

The structure of this paper is presented as follows. Section~\ref{sec:pro} formulates the coverage control problem for MASs in uncertain environments. Section~\ref{sec:main} presents the main results and theoretical proofs. Section~\ref{sec:num} provides experimental validation of the sectorial coverage approach. The conclusion is made in Section~\ref{sec:con}.

\section{Problem Formulation}\label{sec:pro}
Consider a two-dimensional coverage region $\Omega$, whose boundary is described by the equation $L(x, y)=0$. 
Thus, the coverage region can be defined as $\Omega= \{(x, y)\in \mathbb{R}^2|L(x, y) \leq 0\}$, as shown in Fig.~\ref{fig:reg}.
By construction, $\Omega$ is a bounded set and its boundary $\partial\Omega$ consists of finite smooth curve segments.
Multi-agent coverage refers to the strategic deployment of agents within 
a region to effectively deal with random events occurring in that area.
A spatially varying density function $\rho: \Omega\rightarrow\mathbb{R}^+$, constrained within the bounds \([\underline{\rho}, \bar{\rho}]\), is defined to quantify the spatial distribution of event occurrence probabilities over the domain \(\Omega\). Specifically, \(\rho(x, y)\) represents the likelihood that a random event may arise at a particular location \((x, y) \in \Omega\). Accordingly, regions with the higher density correspond to areas where random events are more likely to occur, thereby requiring more attention in the coverage strategy. Consider a group of $N$ agents indexed by the set 
\( \mathbb{I}_N = \{1, 2, \ldots, N\} \), which cooperatively perform coverage 
over the region \( \Omega \). 
The communication topology among agents is modeled by a static undirected graph 
\(\mathcal{G} = (\mathcal{V}, \mathcal{E}) \), 
where \( \mathcal{V} \) denotes the set of vertices corresponding to the agents, 
and \( \mathcal{E} \) represents the set of edges indicating communication links 
between agent pairs. For each agent \( i \in \mathcal{V} \), the neighbor set is defined as $\mathcal{N}_i=\{\, j\in\mathcal{V} \mid (i,j)\in\mathcal{E}\,\}$.
It is assumed that the graph \( \mathcal{G} \) is always connected, 
and each agent can obtain the state information of its neighbors 
through its sensing or communication devices.

To achieve workload balance, each agent employs a virtual rotary pointer, which serves as a partition bar that is able to rotate around its own reference point (i.e., beacon). The subregion assigned to each agent is enclosed by region boundary together with two rotary pointers (see Fig.~\ref{fig:reg}). Each agent is able to detect all events occurring within its own subregion, and it is responsible solely for handling the events that arise within its own subregion.
The rotary pointer of agent \(i\) is described as
\begin{equation}
\mathbf{s}_i=\mathbf{r}_i + \kappa \mathbf{d}_i, \quad i \in \mathbb{I}_N 
\end{equation}
with \(\mathbf{r}_i = (x^{r}_i, y^{r}_i)\) denotes the reference point of the  $i$-th agent. \(\kappa\geq 0\) is a scalar, and \(\mathbf{d}_i=[\cos \varphi_i, \sin\varphi_i]^T\) refers to the unit direction vector associated with the phase angle \(\varphi_i \in [0, 2\pi)\).
\begin{figure}[t!]
\centering
\includegraphics[width=3in]{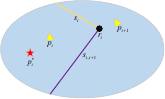}
\caption{Illustration of a  coverage region.  The coverage region is divided into $N$ subregions, each of which is enclosed by a pair of rotary pointers that rotate around the reference point, together with the boundary of the region. The red star indicates the target point $\mathbf{p}_i^{*}$ of subregion $\Omega_i$, while the yellow triangle denotes the position $\mathbf{p}_i$ of agent $i$.
The yellow line segment $\mathbf{s}_i$ and the purple line segment $\mathbf{s}_{i,i+1}$ indicate the two partition bars of agent $i$, while the black dot represents its reference point $\mathbf{r}_i$.}\label{fig:reg}
\end{figure}
Then, the adjacent rotary pointer used by the $i$-th agent for region partition is given by
\begin{equation}
\mathbf{s}_{i,i+1}=\mathbf{r}_i + \kappa \mathbf{d}_{i+1}, \quad i\in \mathbb{I}_N 
\end{equation}
with \(\mathbf{d}_{i+1} = [\cos\varphi_{i+1}, \sin\varphi_{i+1}]^T\) and 
\(\varphi_{i+1}\in [0, 2\pi)\).
Here, \( \varphi_{i+1} \) represents the phase angle of the $(I+1)$-th rotary pointer, which the $i$-th agent obtains from the \((i+1)\)-th agent.
The subregion of the \( i \)-th agent is represented as
$$
\Omega_i=\left\{(x, y) \mid L(x, y) \leq 0 \cap \varphi_i \leq \arctan\frac{y-y_i^{r}}{x - x^{r}_i} \leq \varphi_{i+1}\right\}.
$$
The workload of the \(i\)-th subregion is given by
\begin{equation}\label{eq:m}
m_i = \int_{\Omega_i} \rho(\mathbf{q})\, d\mathbf{q}, \quad i \in \mathbb{I}_N .
\end{equation}
To achieve a complete region partition and workload balance among subregions, the dynamics of reference points and rotary pointers are designed as
\begin{equation}\label{eq:fer}
\begin{split}
\dot{\varphi}_i = -\kappa_{\varphi} & \big[(2m_i - m_{i-1} - m_{i+1}) \frac{\partial m_i}{\partial \varphi_i} \\
& + (2m_{i-1} - m_{i-2} - m_i) \frac{\partial m_{i-1}}{\partial \varphi_i}\big]
\end{split}
\end{equation}
and
\begin{equation}\label{eq:point}
\begin{split}
\dot{\mathbf{r}}_i = -\kappa_r&\big[\,
(2m_i-m_{i-1}-m_{i+1})\frac{\partial m_i}{\partial\mathbf{r}_i}
\\
& + (2\mathbf{r}_i-\mathbf{r}_{i-1}-\mathbf{r}_{i+1})^T			
\big]^T
\end{split}
\end{equation}
with $\kappa_{\varphi}>0$ and $\kappa_r>0$. Note that the neighbor set of the $i$-th agent is defined as $\mathcal{N}_i=\{i+1, i-1, i-2\}$ with \(m_0 = m_N\) and \(\mathbf{r}_0=\mathbf{r}_N\), and the periodicity of hase angle is maintained by $\varphi_i=\mathrm{mod}(\varphi_i, 2\pi)$. Inspired by the work in~\cite{b4}, the performance index of multi-agent coverage is given by
\begin{equation}\label{eq:11}
J(\boldsymbol{\varphi}, \mathbf{p},\mathbf{r}) 
= \sum_{i=1}^N \int_{\Omega_i} f(\mathbf{p}_i, \mathbf{q}) \rho(\mathbf{q}) d\mathbf{q}
\end{equation}
with \(\mathbf{p} = (\mathbf{p}_1, \mathbf{p}_2, \ldots, \mathbf{p}_N)^T\), \(\mathbf{r} = (\mathbf{r}_1, \mathbf{r}_2, \ldots, \mathbf{r}_N)^T\)  and \(\boldsymbol{\varphi} = (\varphi_1, \varphi_2, \ldots, \varphi_N)^T\). In addition, \( f(\mathbf{p}_i, \mathbf{q})\) measures the cost of servicing \(\mathbf{q}\) from \(\mathbf{p}_i\), and it is differentiable with respect to 
\(\mathbf{p}_i\). The dynamics of the \(i\)-th agent is represented by
\begin{equation}\label{eq:agent}
\dot{\mathbf{p}}_i=\mathbf{u}_i, \quad i \in \mathbb{I}_N 
\end{equation}
with the control input \( \mathbf{u}_i \). The performance index \eqref{eq:11} measures the total cost for MAS in dealing with random events in the region \( \Omega \).  
Essentially, this work aims to develop a distributed control algorithm for multi-agent system \eqref{eq:agent} by solving the following optimization problem
\begin{equation}
\begin{aligned}
& \min_{\varphi, \mathbf{p},\mathbf{r}} J(\varphi, \mathbf{p},\mathbf{r}) \\
\text{s.t.} \quad & m_i = m_j, \quad \mathbf{r}_i=\mathbf{r}_j, \quad \forall~i, j \in \mathbb{I}_N .
\end{aligned}
\end{equation}

\section{Main Result}\label{sec:main}
This section focuses on the convergence analysis of reference points and virtual rotary pointers via the consensus-based approach. Then, a distributed control strategy is proposed to guide each agent to the desired equilibrium while reducing the overall coverage cost. 

\subsection{Stability Analysis for Region Partition}\label{A}
To facilitate the stability analysis, the following theoretical result is introduced to characterize the convergence behavior of partition dynamics.
\begin{lemma}\label{lem:seteq}
$$
\mathcal{S}(\mathbf{m}, \mathbf{r}) = \mathcal{M}(\mathbf{m}, \mathbf{r}),
$$
where
$$
\mathcal{S}(\mathbf{m}, \mathbf{r}) = \left\{ (\mathbf{m}, \mathbf{r}) \;\middle|\;\dot{\varphi}_i=0,\ \dot{\mathbf{r}}_i=\mathbf{0},\ \forall i \in \mathbb{I}_N \right\}
$$
and 
\begin{equation}
\mathcal{M}(\mathbf{m},\mathbf{r}) = \left\{ (\mathbf{m},\mathbf{r} )\mid m_i = m_j, \mathbf{r}_i=\mathbf{r}_j, \forall i, j \in\mathbb{I}_N \right\}.
\end{equation}  
\end{lemma}
\begin{proof}
According to Eq. (1) in \cite{b14}, it follows that
\begin{equation*}
\frac{\partial m_i}{\partial \varphi_i}=-\int_{\mathbf{s}_i(\varphi_i)} \rho(\mathbf{q})\, d\mathbf{q}
\end{equation*}
and
\begin{equation*}
\frac{\partial m_i}{\partial \varphi_{i+1}}=\int_{\mathbf{s}_{i,i+1}(\varphi_{i+1})} \rho(\mathbf{q})\, d\mathbf{q}.
\end{equation*}
Since \(\rho(\mathbf{q}) > 0\) for all \(\mathbf{q}\in\Omega\), one gets
$$
\frac{\partial{m}_i}{\partial\varphi_i}< 0, \quad \frac{\partial m_i}{\partial \varphi_{i+1}} > 0, \quad\forall i \in \mathbb{I}_N.
$$
The condition $\dot{\varphi}_i=0$ is equivalent to
\begin{equation*}
\begin{aligned}
&\left(2m_i - m_{i-1} - m_{i+1}\right) \frac{\partial m_i}{\partial \varphi_i} + \\
&\left(2m_{i-1} - m_{i-2} - m_i\right) \frac{\partial m_{i-1}}{\partial \varphi_i}=0,
\end{aligned}
\end{equation*}
which leads to
\begin{equation*}
\Delta m_i = \zeta_i \Delta m_{i-1}
\end{equation*}
with 
$$ 
\zeta_i = -\dfrac{{\partial m_{i-1}}/{\partial \varphi_i}}{{\partial m_i}/{\partial \varphi_i}} > 0
$$
and $\Delta{m}_i=2m_i-m_{i-1}-m_{i+1}$, $\forall~i\in \mathbb{I}_N$.
Note that
$\Delta m_1 = \zeta_1\Delta m_{N}$ for for \(i = 1\) due to the periodicity of ring-like network topology. Thus, we can conclude that
\begin{equation*}
\Delta m_1 = \prod_{i=1}^{N} \zeta_i \Delta m_{1}.
\end{equation*}
For $\prod_{i=1}^{N}\zeta_i\neq 1$,
one gets \(\Delta{m}_1 = 0\), which implies \(\Delta{m}_i = 0\) and 
$m_i= m_{i+1}$, $\forall i\in\mathbb{I}_N$.
For $\prod_{i=1}^{N} \zeta_i = 1$,
considering that 
\begin{equation*}
\begin{aligned}
\sum_{i=1}^{N} \Delta m_i 
&= \Delta m_1 + \left( \prod_{j=1}^{2} \zeta_j \Delta m_1 \right) + \cdots + \\
&\quad \left( \prod_{j=1}^{i} \zeta_j \Delta m_1 \right) + \cdots + \left( \prod_{j=1}^{N} \zeta_j \Delta m_1 \right) \\
&= \left[ 1 + \prod_{j=1}^{2} \zeta_j + \prod_{j=1}^{3} \zeta_j + \cdots + \prod_{j=1}^{N} \zeta_j \right]\Delta m_1  \\
&= \left[ 1 + \sum_{i=2}^{N} \prod_{j=1}^{i} \zeta_j \right]\Delta m_1  = 0
\end{aligned}
\end{equation*}
and \(\zeta_j > 0\), one gets $\Delta{m}_1=0$ and thus \(\Delta m_i = 0\) due to $\Delta m_i = \zeta_i \Delta m_{i-1}$, \(\forall~i\in \mathbb{I}_N\). 
This implies \(m_i=m_{i+1}, \forall~i \in \mathbb{I}_N\).
Thus, it follows that $\dot{\varphi}_i=0$ leads to \(m_i = m_j , \forall i,j \in \mathbb{I}_N\).
By substituting \(m_i = m_j, \forall i \in \mathbb{I}_N\) into \eqref{eq:point}, one gets
$2\mathbf{r}_{i}-\mathbf{r}_{i+1}-\mathbf{r}_{i-1} = \mathbf{0}$
on the condition of $\dot{\mathbf{r}}_i=\mathbf{0}$.
Define $\Delta \mathbf{r}_i = \mathbf{r}_i - \mathbf{r}_{i-1}$, which leads to $\Delta \mathbf{r}_{i+1} = \Delta \mathbf{r}_i=\mathbf{c}$ and
$\mathbf{r}_i = \mathbf{r}_1 + (i-1)\mathbf{c}$, $\forall i \in \mathbb{I}_N$. 
Under the cyclic boundary condition \(\mathbf{r}_{N+1}=\mathbf{r}_1\) , one gets
$\mathbf{r}_{N+1}=\mathbf{r}_1+N\mathbf{c}=\mathbf{r}_1$,
which implies \(N\mathbf{c}=0\) and thus \(\mathbf{c}=0\). 
Thus, it follows that $\dot{\varphi}_i=0$ and $\dot{\mathbf{r}}_i=\mathbf{0}$ leads to 
$\mathbf{r}_i=\mathbf{r}_j$, $\forall i,j \in \mathbb{I}_N$. 
This allows to obtain
$\mathcal{S}(\mathbf{m}, \mathbf{r})\subseteq\mathcal{M}(\mathbf{m},\mathbf{r})$. Moreover, in light of \eqref{eq:fer} and \eqref{eq:point}, the condtion $\mathbf{r}_i=\mathbf{r}_j$ and $m_i=m_j$ results in $\dot{\varphi}_i=0$ and $\dot{\mathbf{r}}_i=\mathbf{0}$, $\forall i,j \in \mathbb{I}_N$, which indicates $\mathcal{M}(\mathbf{m}, \mathbf{r})\subseteq\mathcal{S}(\mathbf{m},\mathbf{r})$. Thus, one gets
$\mathcal{M}(\mathbf{m}, \mathbf{r})=\mathcal{S}(\mathbf{m},\mathbf{r})$. The proof is completed.
\end{proof}

\begin{lemma}\label{lem:con}
Partition dynamics \eqref{eq:fer} and \eqref{eq:point} enables the state trajectory of vector pair $(\mathbf{m},\mathbf{r})$ converging into 
the set \( \mathcal{M}(\mathbf{m},\mathbf{r})\).
\end{lemma}

\begin{proof}
According to the Leibniz integral rule~\cite{b16}, the time derivative of $m_i$ is given by
\begin{align*}
\dot{m}_i &= \frac{\partial m_i}{\partial \mathbf{r}_i}\dot{\mathbf{r}}_i + \frac{\partial m_i}{\partial \varphi_i} \dot{\varphi}_i+\frac{\partial m_i}{\partial \varphi_{i+1}}\dot{\varphi}_{i+1} \\
&= \int_{\partial \Omega_i(t)} \rho(\mathbf{q}) \mathbf{n}_i^T \frac{\partial q}{\partial \mathbf{r}_i} d\mathbf{q}\dot{\mathbf{r}}_i + \int_{\partial \Omega_i(t)} \rho(\mathbf{q}) \mathbf{n}_i^T \frac{\partial q}{\partial \varphi_i} d\mathbf{q}\dot{\varphi}_i \\&+ \int_{\partial \Omega_i(t)} \rho(\mathbf{q}) \mathbf{n}_i^T \frac{\partial q}{\partial \varphi_{i+1}} d\mathbf{q}\dot{\varphi}_{i+1},
\end{align*}
where \(\partial \Omega_i\) represents the boundary of the region \(\Omega_i\), and \(\mathbf{n}_i\)  denotes the unit outward normal vectors on \(\partial \Omega_i\).
Construct the following candidate Lyapunov function
\begin{equation*}
V =\frac{1}{2}\sum_{i=1}^N (m_i- m_{i+1})^2 + \frac{1}{2} \sum_{i=1}^N \left\|\mathbf{r}_i-\mathbf{r}_{i+1}\right\|^2.
\end{equation*}  
Then Its time derivative can be expressed as
\begin{align}
\dot{V} 
&= \sum_{i=1}^N (m_i - m_{i+1})(\dot{m}_i - \dot{m}_{i+1}) \notag\\
&+ \sum_{i=1}^N (\mathbf{r}_i - \mathbf{r}_{i+1})^T (\dot{\mathbf{r}}_i - \dot{\mathbf{r}}_{i+1}) \notag\\
&= \sum_{i=1}^N (2m_i - m_{i-1} - m_{i+1}) \dot{m}_i \notag\\
&+ \sum_{i=1}^N (2\mathbf{r}_i - \mathbf{r}_{i+1} - \mathbf{r}_{i-1})^T \dot{\mathbf{r}}_i \notag\\
&=\sum_{i=1}^N \Big[ 
(2m_i - m_{i-1} - m_{i+1}) \frac{\partial m_i}{\partial \varphi_i} \notag\\
&+ (2m_{i-1} - m_{i-2} - m_i) \frac{\partial m_{i-1}}{\partial \varphi_i} 
\Big] \dot{\varphi}_i \notag\\
&+ \sum_{i=1}^N \Big[ 
(2m_i - m_{i-1} - m_{i+1}) \frac{\partial m_i}{\partial \mathbf{r}_i}\notag\\
&+ (2\mathbf{r}_i - \mathbf{r}_{i+1} - \mathbf{r}_{i-1})^T
\Big] \dot{\mathbf{r}}_i \notag\\
&= -k_{\varphi} \sum_{i=1}^N \Big( 
(2m_i - m_{i-1} - m_{i+1}) \frac{\partial m_i}{\partial \varphi_i} \notag\\
&+ (2m_{i-1} - m_{i-2} - m_i) \frac{\partial m_{i-1}}{\partial \varphi_i} 
\Big)^2 \notag\\
&- k_r \sum_{i=1}^N \big\| 
(2m_i - m_{i-1} - m_{i+1}) \frac{\partial m_i}{\partial \mathbf{r}_i} \notag \\ 
&+ (2\mathbf{r}_i - \mathbf{r}_{i+1} - \mathbf{r}_{i-1})^T
\big\|^2\notag \leq 0\notag.
\end{align}
In light of LaSalle invariance principle, the state trajectories of the dynamical system \eqref{eq:fer} and \eqref{eq:point}  will converge to the largest invariant set 
$\mathcal{S}(\mathbf{m},\mathbf{r}) = \{ (\mathbf{m},\mathbf{r}) \mid \dot{V} = 0 \}$. It follows from $\mathcal{S}(\mathbf{m},\mathbf{r}) = \mathcal{M}(\mathbf{m}, \mathbf{r})$ in Lemma~\ref{lem:seteq} that the system converges to the configuration of 
$m_i = m_j$, $\mathbf{r}_i = \mathbf{r}_j$, $\forall~i,j\in \mathbb{I}_N$.
\end{proof}
    
Define the optimal position of each agent in its own subregion as
$$
\mathbf{p}_i^{*}=\arg\inf_{\mathbf{p}_i\in\Pi_i}J_{\Omega_i}
$$
with the local cost function
$J_{\Omega_i} = \int_{\Omega_i} f(\mathbf{p}_i, \mathbf{q})\,\rho(\mathbf{q})\,d\mathbf{q}$
and the admissible set
$\Pi_i=\{\mathbf{p}_i\in\Omega_i\mid\nabla_{\mathbf{p}_i} J_{\Omega_i}=0\}$.
Let \(\mathbf{H}_{J,\mathbf{p}_i}(\mathbf{p}_i^*)\) denote the Hessian matrix of \(J_{\Omega_i}\) with respect to \(\mathbf{p}_i\) evaluated at 
\(\mathbf{p}_i^*\). 

\begin{lemma}\label{lem:opt}
If $\text{rank}[\mathbf{H}_{J,\mathbf{p}_{i}}(\mathbf{p}_{i}^*)]=2$, one has
$\lim_{t\to\infty}\dot{\mathbf{p}}_i^*(t)=\mathbf{0}, \forall~i\in\mathbb{I}_N$.
\end{lemma}

\begin{proof}
The time derivative of \(\nabla_{\mathbf{p}_{i}} J(\mathbf{p}_{i}^*)= 0\) on both sides leads to
\begin{align*}
\frac{d}{dt} \nabla_{\mathbf{p}_{i}} J(\mathbf{p}_{i}^*) &= \mathbf{H}_{J,\mathbf{p}_{i}}(\mathbf{p}_{i}^*) \dot{\mathbf{p}}_i^* + \frac{\partial \nabla_{\mathbf{p}_{i}} J\mathbf{p}_{i}^*)}{\partial \mathbf{r}_i} \dot{\mathbf{r}}_i \\
&\quad + \frac{\partial \nabla_{\mathbf{p}_{i}} J(\mathbf{p}_{i}^*)}{\partial \varphi_{i+1}} \dot{\varphi}_{i+1} + \frac{\partial \nabla_{\mathbf{p}_{i}} J(\mathbf{p}_{i}^*)}{\partial\varphi_i} \dot{\varphi}_i 
= 0.
\end{align*}  
If \( \mathbf{H}_{J,\mathbf{p}_{i}}(\mathbf{p}_{i}^*) \) is full rank, one gets
\begin{align*}
\dot{\mathbf{p}}_i^* = - \mathbf{H}_{J,\mathbf{p}_{i}}^{-1}(\mathbf{p}_{i}^*) 
\Bigg(&
\frac{\partial \nabla_{\mathbf{p}_{i}} J(\mathbf{p}_{i}^*)}{\partial \mathbf{r}_i} \dot{\mathbf{r}}_i 
+ \frac{\partial \nabla_{\mathbf{p}_{i}} J(\mathbf{p}_{i}^*)}{\partial \varphi_{i+1}} \dot{\varphi}_{i+1} \notag \\
&+ \frac{\partial \nabla_{\mathbf{p}_{i}} J(\mathbf{p}_{i}^*)}{\partial \varphi_i} \dot{\varphi}_i
\Bigg).
\end{align*}
It follows from Lemma~\ref{lem:con} with \eqref{eq:fer} and \eqref{eq:point} that \( \lim_{t \to \infty} \dot{\varphi}_i(t) = 0 \) and \( \lim_{t \to \infty} \dot{\mathbf{r}}_i(t) = \mathbf{0} \). 
By applying the Leibniz rule to parameter-dependent domains, the partial derivatives of
\(\nabla_{\mathbf{p}_i} J(\mathbf{p}_i)\)  evaluated at \(\mathbf{p}_i^*\) can be described as follows 
\[
\frac{\partial \nabla_{p_i} J(\mathbf{p}_i^{*})}{\partial \mathbf{r}_i}
= \int_{\partial \Omega_i} \nabla_{p_i} f\big(\mathbf{p}_i^{*},\mathbf{q}\big)\,\rho(\mathbf{q})\,
 \mathbf{n}_i(\mathbf{q})^\top \tfrac{\partial \mathbf{q}}{\partial \mathbf{r}_i} \, d\mathbf{q},
\]
\[
\frac{\partial \nabla_{p_i} J(\mathbf{p}_i^{*})}{\partial \varphi_i}
= \int_{\partial \Omega_i} \nabla_{p_i} f\big(\mathbf{p}_i^{*},\mathbf{q}\big)\,\rho(\mathbf{q})\,
 \mathbf{n}_i(\mathbf{q})^\top \tfrac{\partial \mathbf{q}}{\partial \varphi_i} \, d\mathbf{q}
\]
and
\[
\frac{\partial \nabla_{p_i} J(\mathbf{p}_i^{*})}{\partial \varphi_{i+1}}
= \int_{\partial \Omega_i} \nabla_{p_i} f\big(\mathbf{p}_i^{*},\mathbf{q}\big)\,\rho(\mathbf{q})\,
 \mathbf{n}_i(\mathbf{q})^\top \tfrac{\partial\mathbf{q}}{\partial \varphi_{i+1}} \, d\mathbf{q},
\]
where
\(\mathbf{n}_i(\mathbf{q})\) is the outward unit normal at \(q\in\partial\Omega_i\).
The terms \({\partial\mathbf{q}}/{\partial \mathbf{r}_i}\), \({\partial\mathbf{q}}/{\partial \varphi_i}\)
and \({\partial\mathbf{q}}/{\partial\varphi_{i+1}}\) describe how boundary points move
when the reference point \(\mathbf{r}_i\) or the phase angles \(\varphi_i\) and \(\varphi_{i+1}\) vary.
On the partition segment $\mathbf{s}_i=\mathbf{r}_i+\kappa\mathbf{d}_i$, one has
\[
\frac{\partial \mathbf{q}}{\partial \mathbf{r}_i}=\mathbf{I}_2,\qquad
\frac{\partial \mathbf{q}}{\partial\varphi_i}=\kappa\begin{bmatrix}-\sin\varphi_i\\\cos\varphi_i\end{bmatrix},
\]
and $\partial\mathbf{q}/\partial\varphi_{i+1}=0$ on $\mathbf{s}_i$. 
The same results hold on $\mathbf{s}_{i,i+1}$. 
In addition, each integrand in the boundary integrals is uniformly bounded on \(\partial\Omega_i\). Since \(\partial\Omega_i\) has finite length, the boundary integrals are finite. Thus, ${\partial \nabla_{p_i} J(\mathbf{p}_i^{*})}/{\partial \mathbf{r}_i}$,
${\partial \nabla_{p_i} J(\mathbf{p}_i^{*})}/{\partial \varphi_i}$ and
${\partial \nabla_{p_i} J(\mathbf{p}_i^{*})}/{\partial \varphi_{i+1}}$ are all bounded, which indicates $\lim_{t\to\infty}\dot{\mathbf{p}}_i^*(t)=\mathbf{0}$. The proof is thus completed.
\end{proof}

\subsection{Design of Coverage  Controller}\label{B}        
In the following, a distributed control strategy is proposed to guide each agent toward a locally optimal position that collectively reduces the global coverage cost. 

\begin{theorem}\label{1}
For multi-agent system~\eqref{eq:agent} with partition dynamics \eqref{eq:fer} and \eqref{eq:point}, the control law
\begin{equation}\label{control}  
\mathbf{u}_i = -\kappa_p \cdot (\mathbf{p}_i - \mathbf{p}_i^*),  \quad \kappa_p>0
\end{equation}
leads to $\lim_{t\to\infty} \|\mathbf{p}_i(t)- \mathbf{p}_i^*(t)\|=0$, $\forall~i\in\mathbb{I}_N$.
\end{theorem}

\begin{proof}
Define the position error $\mathbf{e}_i(t)=\mathbf{p}_i(t)-\mathbf{p}_i^*(t). $ From the agent dynamics \eqref{eq:agent} with the control law \eqref{control} we have 
$\dot{\mathbf{p}}_i(t)=-\kappa_p \mathbf{e}_i(t), $
 hence the error dynamics is 
 \[
 \dot{\mathbf{e}}_i(t)=\dot{\mathbf{p}}_i(t)-\dot{\mathbf{p}}_i^*(t)=-\kappa_p\mathbf{e}_i(t)-
 \dot{\mathbf{p}}_i^*(t). 
 \] 
Integrating the error dynamics gives
\[\mathbf{e}_i(t) = \mathbf{e}_i(0)e^{-\kappa_p t}- \int_0^t e^{-\kappa_p (t - \tau)} \dot{\mathbf{p}}_i^*(\tau) \, d\tau. \] 
Since
\begin{align*}
\quad\int_0^t e^{-\kappa_p (t-\tau)} \dot{\mathbf{p}}_i^*(\tau)\, d\tau 
&= e^{-\kappa_pt}\int_0^t e^{\kappa_p\tau} \dot{\mathbf{p}}_i^*(\tau)\, d\tau \notag \\
&=\frac{\int_0^T e^{\kappa_p\tau} \dot{\mathbf{p}}_i^*(\tau)\, d\tau}{e^{\kappa_pt}},
\end{align*} 
one has
\begin{align*}
\lim_{t\to+\infty}\quad\int_0^t e^{-\kappa_p (t-\tau)} \dot{\mathbf{p}}_i^*(\tau)\, d\tau & = \lim_{t\to+\infty} \frac{\int_0^t e^{\kappa_p\tau} \dot{\mathbf{p}}_i^*(\tau)\, d\tau}{e^{\kappa_pt}}\\
&=\lim_{t\to+\infty}\frac{\dot{\mathbf{p}}_i^*(t)}{\kappa_p}=\mathbf{0}
\end{align*}
according to Lemma~\ref{lem:opt} and L'Hopital's rule.
Considering that $\lim_{t\to+\infty}\mathbf{e}_i(0)e^{-\kappa_p t}=\mathbf{0}$, one gets $\lim_{t\to\infty}\mathbf{e}_i(t)=\mathbf{0}$, which implies $\lim_{t\to\infty}\|\mathbf{p}_i(t)-\mathbf{p}_i^*(t) \|=0$, 
$\forall i\in\mathbb{I}_N$. This completes the proof.
\end{proof}

The above results can be employed to solve the classic location optimization problem~\cite{b4}, where the function 
$f(\mathbf{p}_i,\mathbf{q}) = \|\mathbf{p}_i-\mathbf{q}\|^2$.  
According to \eqref{eq:11}, this allows to get the following performance index
\begin{equation*}
J(\boldsymbol{\varphi}, \mathbf{p},\mathbf{r}) = \sum_{i=1}^N \int_{\Omega_i}\|\mathbf{p}_i-\mathbf{q}\|^2\rho(\mathbf{q}) \, d\mathbf{q}  
\end{equation*}
For the above location optimization problem, one can obtain two corollaries as follows. 

\begin{corollary}\label{lemma4}
For the partition dynamics \eqref{eq:fer} and \eqref{eq:point} , one gets $\lim_{t \to \infty} \dot{c}_{\Omega _i}(t)=\mathbf{0}, \forall i \in \mathbb{I}_N$.
\end{corollary}

\begin{proof}
According to Lemma~\ref{lem:con}, one has
$\lim_{t \to \infty} \dot{\varphi}_i(t)=0$,  
$\lim_{t \to \infty} \dot{\mathbf{r}}_i(t)=\mathbf{0}$,  $\forall~i\in \mathbb{I}_N$.
Then the time derivative of \(m_i\) is expressed as
\[
\dot{m}_i = \frac{\partial m_i}{\partial \mathbf{r}_i} \dot{\mathbf{r}}_i + \frac{\partial m_i}{\partial \varphi_i} \dot{\varphi}_i + \frac{\partial m_i}{\partial \varphi_{i+1}} \dot{\varphi}_{i+1}.
\]
Since the coverage region $\Omega_i(t)$ is bounded and the integrand 
$\rho(\mathbf{q})\,\mathbf{n}_i^\top \frac{\partial q}{\partial \mathbf{r}_i}$ 
is bounded over $\partial \Omega_i(t)$, the integral
\[
\frac{\partial m_i}{\partial \mathbf{r}_i} = \int_{\partial \Omega_i(t)} \rho(\mathbf{q})\,\mathbf{n}_i^\top \frac{\partial q}{\partial \mathbf{r}_i}\,d\mathbf{q}
\]
is bounded. Similarly, the partial derivatives with respect to the phase angles
${\partial m_i}/{\partial\varphi_i}$ and
${\partial m_i}/{\partial\varphi_{i+1}}$,
are also bounded. Thus, one obtains $\lim_{t\to\infty} \dot{m}_i(t)=0$.
For simplicity, introduce the notation
\[
m_{i,x} = \int_{\Omega_i} x \, \rho(\mathbf{q}) \, d\mathbf{q},
\]
which leads to
\[
\dot{m}_{i,x} = \frac{\partial m_{i,x}}{\partial \mathbf{r}_i} 
\dot{\mathbf{r}}_i + \frac{\partial m_{i,x}}{\partial \varphi_i} \dot{\varphi}_i + \frac{\partial m_{i,x}}{\partial \varphi_{i+1}}\dot{\varphi}_{i+1}
\]
and
$\lim_{t \to \infty} \dot{m}_{i,x}(t) = 0$.
Thus, it follows that
\[
\lim_{t \to \infty} \dot{c}_{\Omega_i}^x(t) = \lim_{t \to \infty} \frac{\dot{m}_{i,x}(t) \cdot m_i(t) - m_{i,x}(t) \cdot \dot{m}_i(t)}{(m_i(t))^2} = 0.
\]
Likewise, one gets $\lim_{t\to\infty}\dot{c}_{\Omega_i}^y(t)=0$, which implies $\lim_{t\to\infty}\dot{c}_{\Omega_i}(t)=\mathbf{0},  \forall~i\in \mathbb{I}_N$. The proof is completed.
\end{proof}

\begin{corollary}
Multi-agent dynamics \eqref{eq:agent} with control input \eqref{control} ensures $\lim_{t \to \infty} \mathbf{p}_i(t) = c_{\Omega_i}^*$, $\forall~i\in \mathbb{I}_N$. 
\end{corollary}

\begin{proof}
Since $\rho(\mathbf{q})>0$, $\forall~\mathbf{q}\in\Omega_i$, 
$m_i=\int_{\Omega_i}\rho(\mathbf{q})\,d\mathbf{q}$
is strictly positive, and the centroid of $\Omega_i$ is given by
\[
c_{\Omega_i}=\frac{\int_{\Omega_i} \mathbf{q}\,\rho(\mathbf{q})\,d\mathbf{q}}{\int_{\Omega_i}\rho(\mathbf{q})\,d\mathbf{q}}.
\]
In addition, the gradient of $J_{\Omega_i}$ with respect to $\mathbf{p}_i$ is presented as
\[
\nabla_{\mathbf{p}_i} J_{\Omega_i}
= 2\int_{\Omega_i} \rho(\mathbf{q})\,(\mathbf{p}_i-\mathbf{q})^{T} d\mathbf{q}
= 2m_i\bigl(\mathbf{p}_i - c_{\Omega_i}\bigr).
\]
Due to $m_i>0$, the condition $\nabla_{\mathbf{p}_i} J_{\Omega_i}=0$ is equivalent to
$\mathbf{p}_i=c_{\Omega_i}$. Moreover, the Hessian of $J_{\Omega_i}$ with respect to $\mathbf{p}_i$ is $\mathbf{H}_{J,\mathbf{p}_i}(\mathbf{p}_i)=2m_i\mathbf{I}_2$,
which is positive definite. Hence $\mathbf{p}_i=c_{\Omega_i}$ is the unique minimizer of $J_{\Omega_i}$, and one has
$\mathbf{p}_i^* = c_{\Omega_i}$.
According to Theorem~\ref{1}, the control law~\eqref{control} guarantees
$\lim_{t\to\infty} \|\mathbf{p}_i(t) - \mathbf{p}_i^*(t)\| = 0$, which leads to
$\lim_{t\to\infty} \mathbf{p}_i(t) = c_{\Omega_i}$, $\forall i\in\mathbb{I}_N$.
This completes the proof.
\end{proof}


\section{Numerical Simulations}\label{sec:num}
This section presents a numerical example to validate the theoretical results. Consider the the function $f(\mathbf{p}_i, \mathbf{q}) = \|\mathbf{p}_i - \mathbf{q}\|^2$, and
the coverage environment is described as 
a closed elliptical region with boundary
${x^2}/{25} +{y^2}/{9} = 1$. The spatial density is specified by
$\rho(x, y) = 10^{-4}( \exp\!\big(\sin^2(\arctan(y/x))+\cos(\arctan(y/x))\big) +\sqrt{x^2+y^2})$.
which is strictly positive in the ellipse.
Six agents are deployed in this environment to execute the coverage task using the proposed partition dynamics and coverage control method. The control gains are set to
\(\kappa_p = 0.04 \), \(\kappa_\varphi = 0.045 \), \( \kappa_r = 0.05 \).
Each agent moves towards its own centroid while executing the partition strategy of rotary pointer to balance the workload. 
The centroid evolution is influenced by the positions of reference points and phase angles of the rotary pointers. 
As shown in Fig.~\ref{fig:main}, red points denote the mobile agents, and
black stars represent the centroids of the subregions. Blue asterisks denote the reference points, and
the curves with different colors correspond to the trajectories of the respective mobile agents.
\begin{figure}[t!]
\centering 
\subfigure{\includegraphics[width=0.23\textwidth,trim=0 0 0 30,clip,valign=t]{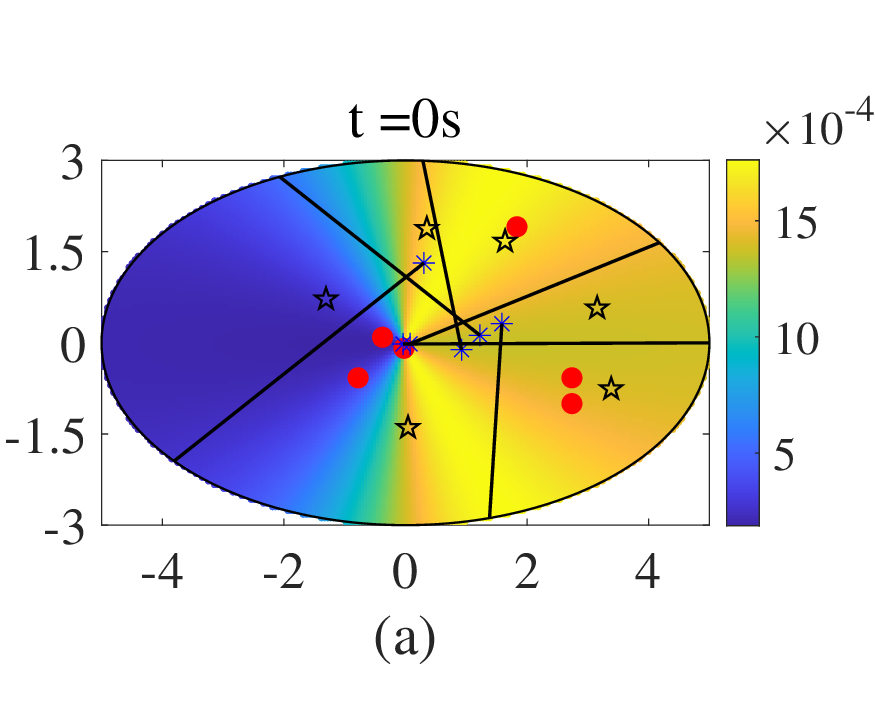}\label{fig:sub1}}%
    \hfill
\subfigure{\includegraphics[width=0.23\textwidth,trim=0 0 0 30,clip,valign=t]{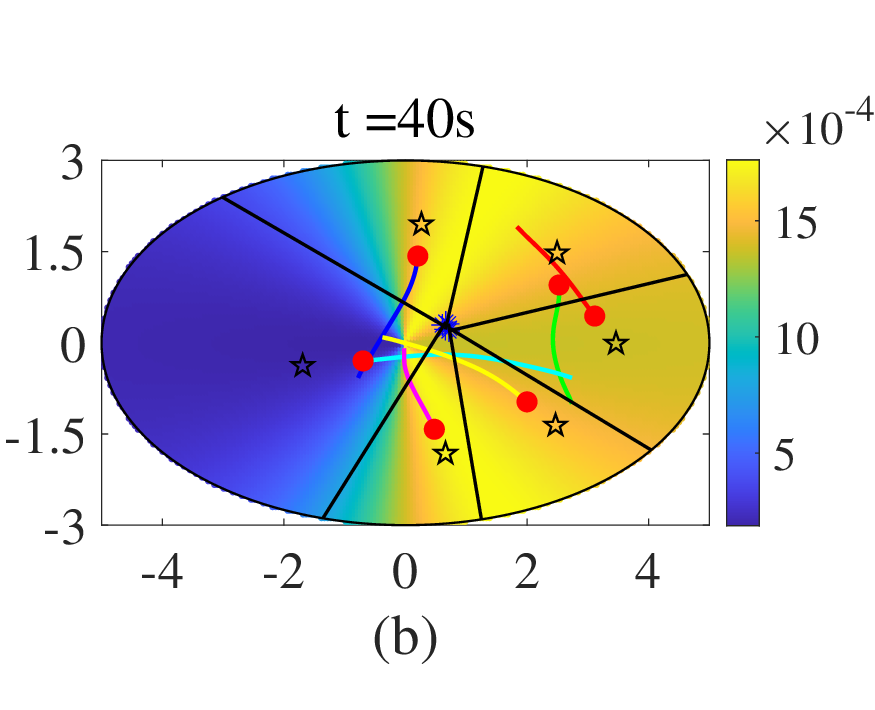}\label{fig:sub2}}%
    \vspace{-4mm}  
\subfigure{\includegraphics[width=0.23\textwidth,trim=0 20 0 30,clip,valign=t]{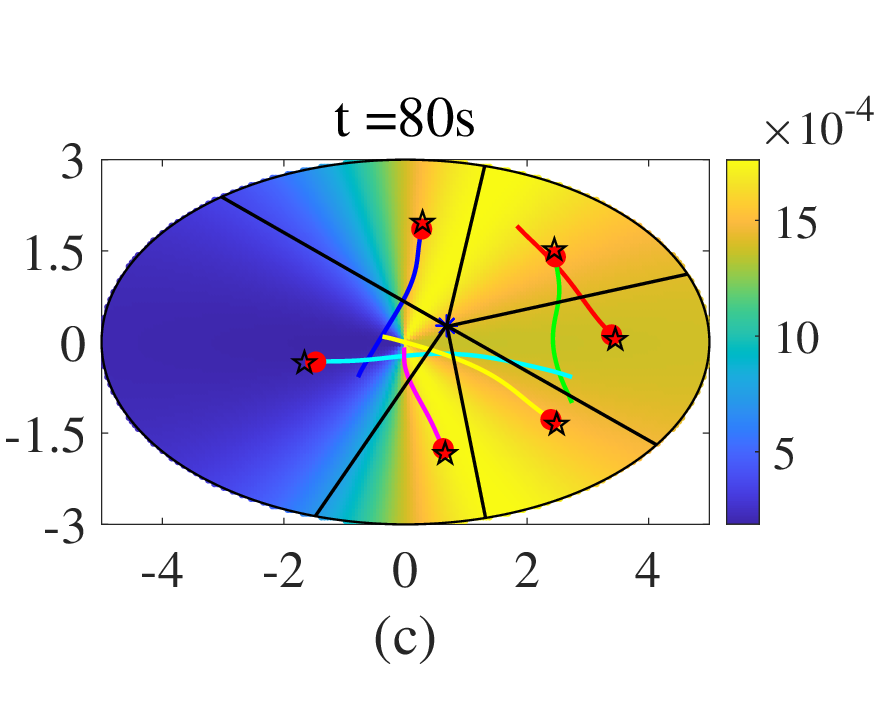}\label{fig:sub3}}%
    \hfill
\subfigure{\includegraphics[width=0.23\textwidth,trim=0 20 0 30,clip,valign=t]{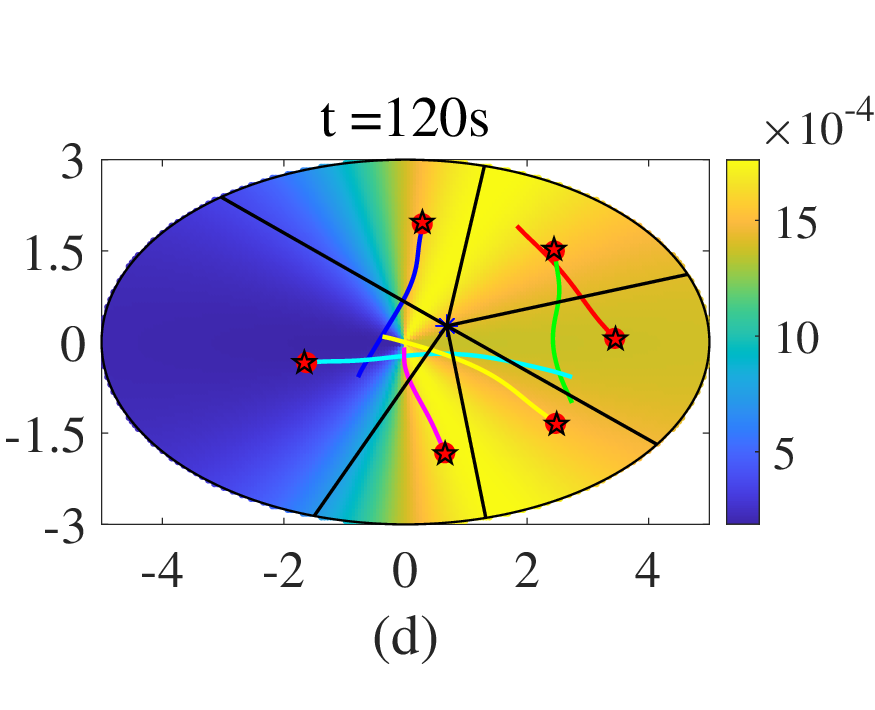}\label{fig:sub4}}%
\caption{Snapshots of simulation results, showing the region partitions, subregion centroids, and agent positions. The mobile agents are marked by red points, the centroids of the subregions are represented by black stars, and the reference points are indicated by blue asterisks.}\label{fig:main}
\end{figure}
In this study, the initial positions of all agents, reference points, and phase angles of the rotary pointers are randomly initialized. 
At \( t=40\,\text{s} \), although the agents have not yet reached the centroids of their respective subregions, their reference points have nearly realized the consensus, thereby completing the region partition. By \( t=80\,\text{s} \), the reference points have almost converged to a common point, and the workload has been nearly equalized. The positions of agents are very close to the centroids of their corresponding subregions. As shown in Fig.~\ref{fig:main}, each agent eventually moves to the centroid of its respective subregion, and the area of each subregion is determined by its associated density value.
To provide a clearer illustration of the algorithm performance, the function $\gamma_i =\|\mathbf{r}_i-\mathbf{r}_{i+1}\|^2$, $\forall~i\in\mathbb{I}_N$ is introduced to quantify the distance between adjacent reference points. Figure~\ref{fig_3} shows the evolution of $\gamma_i$ and the workload on each subregion. It can be observed that both the reference points and the workload of all subregions achieve the consensus.
This demonstrates that the efficiency of proposed coverage control algorithm.

\section{Conclusion}\label{sec:con}
This paper addressed multi-agent coverage control problem in uncertain environments and proposed a novel coverage formulation that enhances coverage performance through rotary pointer partition.
Based on this formulation, a distributed control algorithm was developed to ensure the convergence of multi-agent system toward the optimal centroid configuration of subregions. Simulation results demonstrated that the proposed method effectively balanced the workload among agents, improved the overall coverage efficiency with strong adaptability to uncertain environments.
\begin{figure}[t!]
\centering
\includegraphics[width=3.5in]{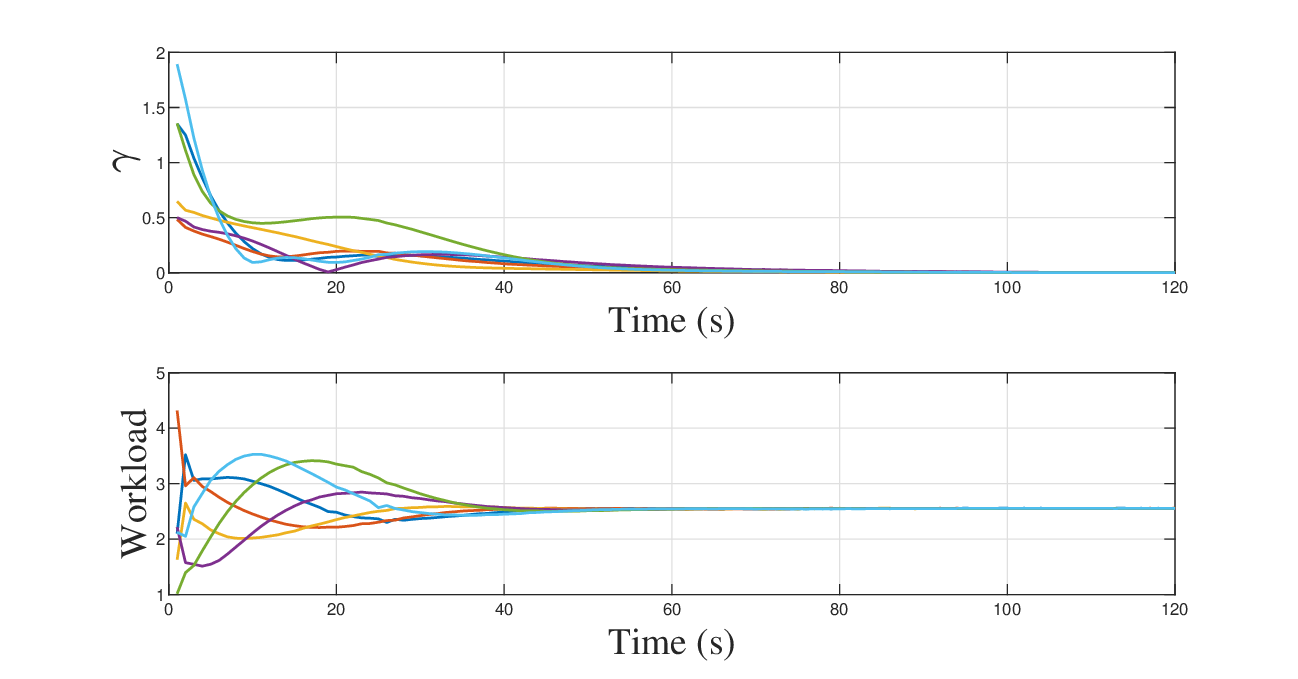}
\caption{Evolution of workload partition and the  $\gamma$ function.}
\label{fig_3}
\end{figure}


\bibliographystyle{IEEEtran}
\bibliography{IEEEexample.bib}

\end{document}